\documentclass[12pt]{amsart}
\usepackage{geometry} 
\usepackage{tikz} 
\usepackage{url}
\usepackage{graphicx}
\usepackage[all]{xy}
\usepackage[mathscr]{eucal}
\usepackage{amsmath,amssymb,amsfonts}
\newtheorem{theorem}{Theorem}[section]
\newtheorem{lemma}[theorem]{Lemma}

\newtheorem{example}[theorem]{Example}

\newtheorem{proposition}[theorem]{Proposition}
\newtheorem{remark}[theorem]{Remark}

\usepackage{stackengine}
\usepackage{xcolor}
\usepackage[all]{xy}

\title{On Ehresmann semigroups}

\author{Mark V. Lawson}
\address{Mark V. Lawson, Department of Mathematics
and the
Maxwell Institute for Mathematical Sciences,
Heriot-Watt University,
Riccarton,
Edinburgh EH14 4AS,
UNITED KINGDOM}
\email{m.v.lawson@hw.ac.uk}

\dedicatory{This paper is dedicated to the memory of Peter M. Neumann}

\thanks{I would like to thank Prof. V. A. R. Gould of the University of York for her comments on an earlier draft of this paper. I would also like to thank the anonymous referee for their careful reading of the submitted paper and helpful suggestions.}

\begin{document}

\begin{abstract}
We formulate an alternative approach to describing Ehresmann semigroups
by means of left and right \'etale actions of a meet semilattice on a category. 
We also characterize the Ehresmann semigroups that arise as the set of all subsets of a finite category.
As applications, we prove that every restriction semigroup can be nicely embedded
into a restriction semigroup constructed from a category,
and we describe when a restriction semigroup can be nicely embedded into an inverse semigroup. 
\end{abstract}
\maketitle

\section{Introduction}

Ehresmann semigroups were introduced in \cite{Lawson1991} as generalizations of inverse semigroups.
We recall their definition here.
An {\em Ehresmann semigroup}
is a semigroup $S$ with a distinguished subset $U \subseteq \mathsf{E}(S)$ of the set of all idempotents,
called the set of {\em projections}, equipped with two functions $\ast, + \colon S \rightarrow U$ satisfying the following four axioms:
\begin{description}
\item[{\rm (ES1)}] $U$ is a commutative subsemigroup; we may therefore, equivalently, view $U$ as a meet semilattice.
\item[{\rm (ES2)}] The maps $\ast$ and $+$ are the identity on $U$. 
\item[{\rm (ES3)}] $aa^{\ast} = a$ and $a^{+} a = a$ for all $a \in S$.
\item[{\rm (ES4)}] $(a^{\ast}b)^{\ast} = (ab)^{\ast}$ and $(ab^{+})^{+} = (ab)^{+}$ for all $a,b \in S$.
\end{description}
Observe that $a^{\ast}$ is the smallest projection $e$ such that $ae = a$;
a dual result holds for $a^{+}$.
Ehresmann semigroups will be denoted by $(S,U)$ to make the set of projections clear or by $S$ alone if the set of projections
is already evident.
A {\em morphism} of Ehresmann semigroups $(S,U)$ and $(T,V)$ is a semigroup homomorphism $\theta \colon S \rightarrow T$
such that $\theta (a)^{\ast} = \theta (a^{\ast})$ and $\theta (a)^{+} = \theta (a^{+})$ 
for all $a \in S$.
Observe that if $e \in U$ then $\theta (e)^{\ast} = \theta (e^{\ast}) = \theta (e)$.
It follows that $\theta$ maps $U$ to $V$.
We say that such a morphism is an {\em isomorphism} if $\theta$ is also an isomorphism of semigroups.
Observe that such an isomorphism induces an isomorphism between $U$ and $V$.
Ehresmann semigroups can be equivalently described as algebras of type $(2,1,1)$;
see \cite{BGG}.
 The fact that Ehresmann semigroups thereby form a variety means that free objects exist.
 We refer the reader to \cite{Kambites} and \cite{FGG} for more information;
 we also recommend the references to be found in both these papers.

\begin{remark}\label{rem:origins}{\em 
Ehresmann semigroups arose within the York School of semigroup theory led by J. B. Fountain \cite{GL}.
The paper \cite{Fountain} was particularly influential;
in particular, the conditions below which we refer to below as `deterministic' and `codeterministic'
were first introduced in this paper.
The theme of this School became that of determining which properties of (von Neumann)
regular semigroups could be generalized to a non-regular setting. 
There was a particular emphasis on the non-regular
generalizations of inverse semigroups with the focus of attention being 
on the abstract relationship between an element $a$ and its idempotents $a^{-1}a$ and $aa^{-1}$.
Subsequently, these ideas were then developed not only for semigroups but also for categories \cite{CGH2012}.
My own paper \cite{Lawson1991} was written within the York School framework with the goal being 
to describe the most general class of non-regular semigroups that could be regarded as natural generalizations of inverse semigroups.
This was achieved by developing ideas due to Charles Ehresmann and his students.
The key stumbling block became that of order:
in an inverse semigroup, the natural partial order encodes algebraically the order induced by subset-inclusion.
For more general classes of Ehresmann semigroups, this neat relationship between algebra and order does not hold.
Trying to deal with this problem led to the notions of `bideterministic element' and `partial isometry'
which play major r\^oles in this paper.}
\end{remark}

The following are standard results about Ehresmann semigroups
but they are also easy to prove directly.

\begin{lemma}\label{lem:basic-results} Let $S$ be an Ehresmann semigroup with set of projections $U$.
\begin{enumerate}
\item $(ab)^{\ast} \leq b^{\ast}$.
\item $(ab)^{+} \leq a^{+}$.
\item If $U$ has a zero then $a = 0$ if and only if $a^{\ast} = 0$.
\item If $U$ has a zero then $a = 0$ if and only if $a^{+} = 0$.
\end{enumerate}
\end{lemma}

Ehresmann semigroups come equipped with two partial orders which are algebraically defined:
\begin{itemize}
\item $x \leq_{r} y$ if and only if $x = ey$ for some $e \in U$. Observe that $x^{+} \leq y^{+}$.
\item $x \leq_{l} y$ if and only if $x = yf$ for some $f \in U$. Observe that $x^{\ast} \leq y^{\ast}$.
\end{itemize}
Define 
$$\leq \,= \,\leq_{l} \cap \leq_{r}.$$
This order will play a special r\^ole in this paper. 
It is easy to check that  the following hold:
\begin{itemize}
\item $x \leq_{r} y$ if and only if $x = x^{+}y$.
\item $x \leq_{l} y$ if and only if $x = yx^{\ast}$.
\item $x \leq y$ if and only if $x = x^{+}y = yx^{\ast}$.
\end{itemize}
Although these orders generalize the natural partial order on an inverse semigroup (where they are all the same)
they do not share such nice properties.
This is an issue we shall have to confront in Section~3.
However, on the set of projections $U$ they agree and we denote their common order by $\leq$.

Ehresmann semigroups have emerged as an interesting class \cite{BGG, BGGW, EG2020, Stein2016}.
In particular, 
they are closely allied to categories in two ways.
First of all, underlying every Ehresmann semigroup is a category.
The following was proved as \cite[Theorem 3.17]{Lawson1991}.

\begin{proposition}\label{prop:cats} Let $S$ be an Ehresmann semigroup with set of projections $U$.
On the set $S$ define the {\em restricted product} $a \cdot b = ab$ when $a^{\ast} = b^{+}$ and undefined
otherwise. Then $(S,\cdot)$ is a category in which $a^{\ast} = \mathbf{d}(a)$ 
and $b^{+} = \mathbf{r}(b)$.
In addition, for any $x,y \in S$ we have that $xy = (xe) \cdot (ey)$ where $e = x^{\ast}y^{+}$.
\end{proposition} 

Second of all, examples of Ehresmann semigroups directly arise from categories;
the following is a special case of \cite{KL2017}

\begin{example}\label{ex:cats}
{\em Let $C$ be a small category with set of identities $C_{o}$.
We denote the domain and codomain maps on $C$ by $\mathbf{d}$ and $\mathbf{r}$, respectively.
The product $xy$ is defined in the category if and only if $\mathbf{d}(x) = \mathbf{r}(y)$;
in this case, we shall also write $\exists xy$. 
Let $S = \mathsf{P}(C)$ be the set of all subsets of $C$ equipped with the multiplication of subsets.
This is a semigroup.
Put $U = \mathsf{P}(C_{o})$, the set of all subsets of the set of identities.
This set consists of idempotents and is closed under multiplication.
For $A \subseteq C$ define $A^{\ast} = \{\mathbf{d}(a) \colon a \in C \}$
and 
$A^{+} = \{\mathbf{r}(a) \colon a \in C \}$.
We now show that with these definitions $\mathsf{P}(C)$ is an Ehresmann monoid.
Let $E,F \subseteq C_{o}$.
Then $EF = E \cap F$.
It follows that $U$ is a commutative subsemigroup of the set of all idempotents of $S$
so that (ES1) holds.
It is immediate from the definitions that (ES2) holds.
It is evident that (ES3) holds.
It remains to show that (ES4) holds.
By symmetry, it is enough to show that $(A^{\ast}B)^{\ast} = (AB)^{\ast}$ holds.
In a category $ab$ is defined if and only if $\mathbf{d}(a) = \mathbf{r}(b)$.
Thus $ab$ is defined if and only if $\mathbf{d}(a)b$ is defined in which case 
$\mathbf{d}(ab) = \mathbf{d}(\mathbf{d}(a)b)$. 
We have therefore proved that we are dealing with an Ehresmann monoid.
A special case of this construction shows that the monoid of all binary relations on a set $X$,
which we denote by $\mathscr{B}(X)$, is also an Ehresmann monoid;
the category in this special case is the set $X \times X$ where $\mathbf{d}(x,y) = (y,y)$ and $\mathbf{r}(x,y) = (x,x)$
and product $(x,y)(y,z) = (x,z)$.}
\end{example}

In \cite{Lawson1991}, we showed what additional structure a category needed to be equipped with
in order that it arise from an Ehresmann semigroup.
In Section~2 of this paper, we shall prove Theorem~\ref{thm:ehresmann-biactions}, which describes a different (though, obviously, equivalent) way, 
in which this can be accomplished.
It arose from reading \cite[Page 184]{Resende2007}.
In Section~3 of this paper, we shall characterize the Ehresmann monoids of the form $\mathsf{P}(C)$, 
where $C$ is a finite category, in Theorem~\ref{them:TWO}.
This result uses ideas first described in \cite{KL2017}.
In Section~4, we prove two theorems, as Theorem~\ref{them:THREE} and Theorem~\ref{them:FOUR}, about special kinds of Ehresmann semigroups called restriction semigroups.

\section{Ehresmann biactions}

The goal of this section is to formulate a different categorical approach to characterizing Ehresmann semigroups
from the one described in \cite{Lawson1991}.
In fact, our approach can now be seen as a special case of \cite{FK}
although it in fact arose from reading \cite[Page 184]{Resende2007}.
Our goal is to prove a kind of converse of Proposition~\ref{prop:cats}.
To do this, we shall need a class of actions called \'etale actions 
\cite{FS2010, Steinberg2011} though we prefer the term `supported actions'.
Let $S$ be an inverse semigroup and $X$ a set.
Let $p \colon X \rightarrow \mathsf{E}(S)$ be a function.
A {\em supported action} is a left action $S \times X \rightarrow X$, denoted by $(s,x) \mapsto s \cdot x$, such that $p(x)\cdot x = x$ and $p(s \cdot x) = sp(x)s^{-1}$.

\begin{remark}\label{rem:extending}
{\em A feature of \'etale actions is that they can be restricted
and then reconstructed from this restriction.
Observe that $s \cdot x = s \cdot (p(x) \cdot x) = (sp(x)) \cdot x$
and that $(sp(x))^{-1}sp(x) = p(x)s^{-1}s \leq p(x)$.
Thus we may restrict the action to those pairs $(s,x)$ where $s^{-1}s \leq p(x)$
and lose nothing. 
We shall not use this result below but it explains why this new approach using actions 
is equivalent to the one adopted in \cite{Lawson1991}
which uses restrictions and corestrictions.}
\end{remark}

We shall only be interested in supported actions where the acting inverse semigroup is a meet semilattice, in which case the above properties simplify somewhat.
We can now  define an {\em Ehresmann biaction} starting from a category $C$ by the following six axioms:
\begin{description}
\item[{\rm (E1)}] The set of identities $C_{o}$ of $C$ is equipped with the structure of a commutative, idempotent semigroup.
\item[{\rm (E2)}] There are two supported actions:
there is a left action $C_{o} \times C \rightarrow C$ denoted by $(e,a) \mapsto e \cdot a$
such that $\mathbf{r}(a) \cdot a = a$ and $\mathbf{r}(e \cdot a) = e\mathbf{r}(a)$;
there is a right action $C \times C_{o} \rightarrow C$ denoted by $(a, e) \mapsto a \cdot e$
such that $a \cdot \mathbf{d}(a) = a$ and $\mathbf{d}(a \cdot e) = \mathbf{d}(a)e$.
Observe that we do {\em not} assume, for example, that $e \cdot (ab) = (e \cdot a)(e \cdot b)$. See Axiom (E6) below.
\item[{\rm (E3)}]  The biaction property $(e \cdot a) \cdot f = e \cdot (a \cdot f)$ holds.
\item[{\rm (E4)}] We require that $e \cdot a = ea$ and $a \cdot e = ae$ if $a \in C_{o}$.
\item[{\rm (E5)}]  $\mathbf{d}(e \cdot a) \leq \mathbf{d}(a)$ and $\mathbf{r}(a \cdot e) \leq \mathbf{r}(a)$.
\item[{\rm (E6)}]  When $\exists xy$ then
$$e \cdot (xy) = (e \cdot x)(\mathbf{d}(e \cdot x) \cdot y)
\text{ and }
(xy) \cdot e = (x \cdot \mathbf{r}(y \cdot e))(y \cdot e).$$
\end{description}

\begin{remark}\label{rem:biactions}{\em The essential difference between the approach to Ehresmann semigroups
described here and the one described in \cite{Lawson1991} is that
the latter requires orders whereas here we use only actions.
Remark~\ref{rem:extending} provides some insight into the relationship between the two approaches.}\end{remark}

We now prove that every Ehresmann semigroup gives rise to an Ehresmann biaction.

\begin{proposition}\label{prop:ehresmann-biaction} Let $(S,U)$ be an Ehresmann semigroup.
Put $C = S$ equipped with the restricted product.
Then $C_{o} = U$, a commutative idempotent semigroup,
and $C$ is an Ehresmann biaction.
\end{proposition}
\begin{proof} The left action $C_{o} \times C \rightarrow C$ and the right action $C \times C_{o} \rightarrow C$ are
both defined by multiplication.
All the axioms are easy to check from the axioms and properties of an Ehresmann semigroup.
\end{proof}

Proposition~\ref{prop:ehresmann-biaction} tells us that Ehresmann semigroups give rise to Ehresmann biactions. 
We now prove the first theorem of this paper which shows us that we can construct Ehresmann semigroups from Ehresmann biactions.

\begin{theorem}\label{thm:ehresmann-biactions} Let $C$ be an Ehresmann biaction.
Given $x,y \in C$,
put $e = \mathbf{d}(x)\mathbf{r}(y)$
and define $x \bullet y = (x \cdot e)(e \cdot y)$, called the {\em pseudoproduct}.
Then $(C,\bullet)$ is an Ehresmann semigroup with set of projections $C_{o}$.
\end{theorem}
\begin{proof}  The fact that the pseudoproduct is a binary operation follows from axioms (E1) and (E2).
We prove that the pseudoproduct is associative.
First, observe that 
$$(x \bullet y) \bullet z
=
[(x \cdot f) \cdot \mathbf{r}(f \cdot y \cdot e)]
(f \cdot y \cdot e)
(e \cdot z)$$
where $e = \mathbf{d}(x \bullet y) \mathbf{r}(z)$ and $f = \mathbf{d}(x)\mathbf{r}(y)$
and using axioms (E6) and (E3).
It is easy to check from the axioms that the following hold: 
\begin{itemize}
\item $e = \mathbf{d}(f \cdot y) \mathbf{r}(z)$.
\item $f \cdot y = \mathbf{d}(x) \cdot y$ by axiom (E2); where we use the properties of the action and the fact that $\mathbf{r}(y) \cdot y = y$.
\item $f \cdot y \cdot e = \mathbf{d}(x) \cdot y \cdot \mathbf{r}(z)$.
This follows from the fact that
$(f \cdot y) \cdot e = (f \cdot y) \cdot \mathbf{d}(f \cdot y)\mathbf{r}(z)$
which is equal to $(f \cdot y) \cdot \mathbf{r}(z)$ using axiom (E2)
and the result now follows by one of the results above.
\item $e \cdot z = \mathbf{d}(\mathbf{d}(x) \cdot y) \cdot z$.
This uses axiom (E2) and some of the results above.
\item $[(x \cdot f) \cdot \mathbf{r}(f \cdot y \cdot e)] 
= x \cdot \mathbf{r}(y \cdot \mathbf{r}(z))$.
Observe that 
$(x \cdot f) \cdot \mathbf{r}(f \cdot y \cdot e)
 = x \cdot \mathbf{r}(f \cdot y \cdot e)$
using axiom (E2).
By one of the results above and using axiom (E2) we get the result.
\end{itemize}
We have therefore shown that
$$(x \bullet y) \bullet z
=
(x \cdot \mathbf{r}(y \cdot \mathbf{r}(z)))(\mathbf{d}(x) \cdot y \cdot \mathbf{r}(z))(\mathbf{d}(\mathbf{d}(x) \cdot y) \cdot z).$$
Second, observe that 
$$x \bullet (y \bullet z)
=
(x \cdot i)(i \cdot y \cdot j)(\mathbf{d}(i \cdot y \cdot j) \cdot j \cdot z)$$
where $i = \mathbf{d}(x)\mathbf{r}(r \bullet z)$ and $j = \mathbf{d}(y)\mathbf{r}(z)$
and using axioms (E6) and (E3).
It is easy to check from the axioms that the following hold with similar proofs to the ones above:
\begin{itemize}
\item $x \cdot i = x \cdot \mathbf{r}(x \cdot \mathbf{r}(z))$.
\item $i \cdot y \cdot j = \mathbf{d}(x) \cdot y \cdot \mathbf{r}(z)$.
\item $\mathbf{d}(i \cdot y \cdot j) \cdot j \cdot z = \mathbf{d}(\mathbf{d}(x) \cdot y) \cdot z$.
\end{itemize}
We have therefore shown that
$$x \bullet (y \bullet z)
=
(x \cdot \mathbf{r}(y \cdot \mathbf{r}(z)))(\mathbf{d}(x) \cdot y \cdot \mathbf{r}(z))(\mathbf{d}(\mathbf{d}(x) \cdot y) \cdot z).$$
It now follows that $\bullet$ is associative.

Observe that if $e \in C_{o}$ then $a \bullet e = a \cdot e$
and that 
if $e,f \in C_{o}$ then $ef = e \bullet f$.
It follows that $\bullet$ is commutative and idempotent on $C_{o}$.
Define 
$$a^{\ast} = \mathbf{d}(a) \text{ and } a^{+} = \mathbf{r}(a).$$
Then 
$$a \bullet a^{\ast} = a \bullet \mathbf{d}(a) = a \cdot \mathbf{d}(a) = a$$
by (E2).
We calculate 
$$(a^{\ast} \bullet b)^{\ast} = \mathbf{d}(\mathbf{d}(a) \bullet b) = \mathbf{d}(e \cdot b)$$
where $e = \mathbf{d}(a)\mathbf{r}(b)$.
Thus $(a^{\ast} \bullet b)^{\ast} = \mathbf{d}(\mathbf{d}(a) \cdot b)$.
On the other hand 
$$(a \bullet b)^{\ast} = \mathbf{d}(a \bullet b) = \mathbf{d}((a \cdot e)(e \cdot b))$$
where $e = \mathbf{d}(a)\mathbf{r}(b)$.
Thus $(a \bullet b)^{\ast} = \mathbf{d}(e \cdot b) = \mathbf{d}(\mathbf{d}(a) \cdot b)$.
The dual results are proved by symmetry.
Thus $(C,\bullet)$ is an Ehresmann semigroup with set of projections $C_{o}$, as claimed.
\end{proof}

\begin{remark}{\em We could continue and show that with suitably defined maps between
 Ehresmann biactions our construction is functorial.
 However, this is straightforward and will not be carried out here.}
 \end{remark}

\section{A class of finite Ehresmann monoids}

In this section, we shall characterize the finite Ehresmann monoids which are isomorphic to those arising from
finite categories as in Example~\ref{ex:cats}.
We have found \cite[Chapter 1]{Givant} a useful reference in helping us to clarify our thoughts.
We have also used ideas from \cite{KL2017} in an essential way.

We shall need to define two classes of elements within an Ehresmann semigroup.
The first requires only the algebraic structure of an Ehresmann semigroup.
We use the terminology from \cite{CGH2012}. 
An element $a \in S$ of an Ehresmann semigroup is said to be {\em deterministic}
if $ea = a(ea)^{\ast}$ for all $e \in U$;
it is said to be {\em codeterministic} if
$ae = (ae)^{+}a$ for all $e \in U$.
An element is said to be {\em bideterministic} if it is both deterministic and codeterministic.

An Ehresmann semigroup is said to be a {\em birestriction semigroup} if every element is bideterministic; in what follows, we shall usually just say `restriction semigroup' rather than `birestriction semigroup'.
We shall say more about restriction semigroups in Section~4.

We now characterize the bideterministic elements of Ehresmann monoids of the form 
$\mathsf{P}(C)$ where $C$ is a category.

\begin{lemma}\label{lem:det-codet} Let $C$ be a category with $\mathsf{P}(C)$ being its associated Ehresmann monoid.
\begin{enumerate}
\item A non-empty subset $A \subseteq C$ is deterministic if and only if whenever $a,b \in A$ and $\mathbf{d}(a) = \mathbf{d}(b)$
then $\mathbf{r}(a) = \mathbf{r}(b)$.
\item A non-empty subset $A \subseteq C$ is codeterministic if and only if whenever $a,b \in A$ and $\mathbf{r}(a) = \mathbf{r}(b)$
then $\mathbf{d}(a) = \mathbf{d}(b)$.
\end{enumerate}
\end{lemma}
\begin{proof}
We prove (1) since (2) follows by symmetry.
Let $A$ be a non-empty subset of $C$.
Suppose that $a,b \in A$ are such that $\mathbf{d}(a) = \mathbf{d}(b) = e$
but $\mathbf{r}(a) = i$ is different from $\mathbf{r}(b) = j$.
We prove that $A$ is non-deterministic.
The singleton set $\{i\}$ is an element  of $U$.
The product  $\{i\}A$ is an element of $\mathsf{P}(S)$ 
 which contains $a$ but which does not contain $b$.
 On the other hand, $(\{i\} A)^{\ast}$ contains $e$.
It follows that $A(\{i\}A)^{\ast}$ contains both $a$ and $b$.
Thus $A$ cannot be deterministic since $\{i\}A \neq A(\{i\}A)^{\ast}$.
Suppose now that $A$ is such that whenever $a,b \in A$ and $\mathbf{d}(a) = \mathbf{d}(b)$ then $\mathbf{r}(a) = \mathbf{r}(b)$.
We prove that $A$ is deterministic.
Let $E$ be any projection.
Let $x \in A(EA)^{\ast}$.
Then $x = ae$ where $a \in A$ and $e \in (EA)^{\ast}$.
Thus $e = \mathbf{d}(e'a')$ where $e' \in E$ and $a' \in A$.
It follows that $x,a' \in A$ and $\mathbf{d}(x) = \mathbf{d}(a')$.
By assumption, $\mathbf{r}(x) = \mathbf{r}(a')$.
But then $x = e'x$ where $e' \in E$.
\end{proof}

Ehresmann monoids of the form $\mathsf{P}(C)$ come equipped with subset inclusion as an order.
We shall formalize its properties below.
Let $(S,U)$ be an Ehresmann monoid with set of projections $U$
where $1 \in U$.
We shall say that $S$ is {\em Boolean} if $S$ is equipped with a partial order, denoted by
 $\subseteq$, 
with respect to which $(S,\subseteq)$ is a Boolean algebra
--- we denote the top element of $S$ by $t$ ---
such that the following properties hold:
\begin{description}
\item[{\rm (OE1)}] $c (a \cup b) = ca \cup cb$, and dually.
\item[{\rm (OE2)}] If $e \in U$ and $a \subseteq e$ then $a \in U$.
\item[{\rm (OE3)}] If $e,f \in U$ then $e \leq f$ if and only if $e \subseteq f$.
 \item[{\rm (OE4)}] If $a = eb$ then $a \subseteq b$ and if $a = be$ then $a \subseteq b$ where $e \in U$.
 \item[{\rm (OE5)}] $(a \cup b)^{\ast} = a^{\ast} \cup b^{\ast}$, and dually.
\end{description}
In a Boolean Ehresmann monoid, the set of projections is the order-ideal determined by the identity element of the monoid by axiom (OE2).

The proof of the following lemma is immediate from axioms (OE1) and (OE5).

\begin{lemma}\label{lem:order-properties-new} Let $(S,U)$ be Boolean Ehresmann monoid.
Then the following two properties hold:
\begin{enumerate}
\item If $a \subseteq b$ and $c \subseteq d$ then $ac \subseteq bd$.
\item If $a \subseteq b$ then $a^{\ast} \subseteq b^{\ast}$ and $a^{+} \subseteq b^{+}$.
\end{enumerate}
\end{lemma}

The following result assures us that the set of projections of a Boolean Ehresmann monoid is also a Boolean algebra.

\begin{lemma} Let $(S,U)$ be a Boolean Ehresmann monoid.
Then the following hold:
\begin{enumerate}
\item $t^{\ast} = 1 = t^{+}$.
\item $U$ is a Boolean algebra in its own right.
\end{enumerate}
\end{lemma}
\begin{proof} (1) The top element $t$ has the property that $a \subseteq t$ for all $a \in S$.
Thus $a^{\ast} \subseteq t^{\ast}$ by Lemma~\ref{lem:order-properties-new}.
It follows that $t^{\ast}$ is an upper bound for all projections.
Thus $1 \subseteq t^{\ast} \subseteq 1$.
We have proved that $t^{\ast} = 1$.
It follows by symmetry that $t^{+} = 1$.

(2) By axiom (OE3), if $e$ and $f$ are projections then $ef = e \wedge f$.
In a Boolean Ehresmann monoid, we are told that $(S,\subseteq)$ is a Boolean algebra. We prove that $U$ is a Boolean algebra.
Let $e,f \in U$.
By assumption, both $e \wedge f$ and $e \vee f$ exist in $S$.
Since $e \wedge f \subseteq e$ we know by axiom (OE2) that $e \wedge f \in U$.
The fact that $e \vee f \in U$ follows by axiom (OE5) and the definition of an Ehresmann
semigroup; alternatively, $e,f \subseteq 1$ and so $e \vee f \subseteq 1$.
It now follows that $e \vee f \in U$ by axiom (OE2).
Thus $U$ is closed under binary greatest lower bounds and least upper bounds.
Let $e \in U$.
We need to prove that $U$ has complements.
Let $e \in U$.
Then $\bar{e} \in S$.
Thus $\neg e = 1 \wedge \bar{e} \in U$.
We calculate $e \wedge \neg e = e \wedge 1 \wedge \bar{e} = 0$
and $e \vee \neg e = e \vee (1 \wedge \bar{e}) = 1 \wedge t = 1$.
Thus $\neg e$ is the complement in $U$ of $e$. 
\end{proof}
 
The motivation for the above definition comes, of course, from Example~\ref{ex:cats}
as the following lemma shows.
The proof is by routine verification.

\begin{lemma}\label{lem:ordered-e-cat}
Let $C$ be a (finite) category.
Then $\mathsf{P}(C)$ is a Boolean Ehresmann monoid.
\end{lemma}

The following notion first arose in \cite{Resende2007} but was then extended in \cite{KL2017}.
In Boolean Ehresmann monoids,
we shall need a stronger notion than that of a bideterministic element.
Let $S$ be a Boolean Ehresmann monoid.
An element $a \in S$ is said to be a {\em partial isometry} if whenever $b \subseteq a$ then $b \leq a$.
Recall that $b \leq a$ means that $b = ea = af$ for some $e,f \in U$.
The following was proved as \cite[Lemma 2.26]{KL2017}.

\begin{lemma}\label{lem:order-ideal} 
In a Boolean Ehresmann monoid, the set of partial isometries is an order-ideal.
\end{lemma}
\begin{proof} Let $a$ be a partial isometry and let $b \subseteq a$.
Since $a$ is a partial isometry, we have that $b \leq a$.
Thus $b = b^{+}a = ab^{\ast}$.
We prove that $b$ is a partial isometry.
Let $c \subseteq b$. Then $c \subseteq a$.
It follows that $c \leq a$.
Thus $c = c^{+}a = ac^{\ast}$.
By axiom (OE5) applied to $c \subseteq b$,
we deduce that $c^{+} \leq b^{+}$ and $c^{\ast} \leq b^{\ast}$.
Thus  
$c^{+}b = c^{+}(b^{+}a) = c^{+}a = c$
and
$bc^{\ast} = (ab^{\ast})c^{\ast} = ac^{\ast} = c$.
It follows that $c \leq b$, as required.
\end{proof}

The following was proved in \cite[Lemma 2.26]{KL2017}
but we give a direct proof here.

\begin{lemma}\label{lem:cake} In a Boolean Ehresmann monoid, 
every partial isometry is bideterministic.
\end{lemma}
\begin{proof} Let $a$ be a partial isometry.
We prove that it is deterministic; the proof that it is codeterministic follows by symmetry.
Let $e$ be any projection.
Then $ae \subseteq a$ by axiom (OE4).
But $a$ is a partial isometry and so $ae \leq a$.
By definition, $ae = a(ae)^{\ast}$.
This proves that $a$ is deterministic.
\end{proof}

We next describe the partial isometries in the ordered Ehresmann monoids $\mathsf{P}(C)$.

\begin{lemma}\label{lem:det-codet-new} Let $C$ be a category.
Then a non-empty subset $A \subseteq C$ is a partial isometry in $\mathsf{P}(C)$ if and only if the following two conditions hold:
\begin{enumerate}
\item If $a,b \in A$ and $\mathbf{d}(a) = \mathbf{d}(b)$ then $a = b$.
\item If $a,b \in A$ and $\mathbf{r}(a) = \mathbf{r}(b)$ then $a = b$.
\end{enumerate}
\end{lemma}
\begin{proof} 
Let $A$ be a subset that satisfies condition (1) and suppose that $B \subseteq A$.
Clearly, $B \subseteq AB^{\ast}$.
We shall show that $B = AB^{\ast}$.
Let $x \in AB^{\ast}$.
Then $x = a\mathbf{d}(b)$ where $a \in A$ and $b \in B$.
Thus $\mathbf{d}(x) = \mathbf{d}(b)$.
Now, $x,b \in A$.
By the assumption that $A$ satisfies the condition (1),
we must have that $x = b$.
We have therefore proved that $B = AB^{\ast}$.
The fact that $B = B^{+}A$ if $A$ satisfies condition (2) follows by symmetry.
We now prove the converse.
Suppose that $A$ is a partial isometry.
We prove that condition (1) holds;
the fact that condition (2) holds follows by symmetry.
Suppose that $a,b \in A$ be distinct elements such that $\mathbf{d}(a) = \mathbf{d}(b) = e$.
Clearly, $\{a\}, \{b\} \subseteq A$,
but we do not have that $\{a\} \leq A$.
The reason is that $\{a\}^{\ast} = \{e\}$ and $A\{a\}^{\ast} = A\{e\}$ contains both $a$ and $b$.
Thus this set cannot be equal to $\{a\}$.
\end{proof}

\begin{remark}\label{rem:trump}{\em Subsets of groupoids satisfying both the conditions of Lemma~\ref{lem:det-codet-new} 
are called {\em local bisections}.
The set of all local bisections of a groupoid forms an inverse semigroup.
See \cite[page 164]{Resende2007}.}
\end{remark}

\begin{lemma}\label{lem:atoms-are-pi} Let $S$ be a Boolean Ehresmann monoid.
Then any atom is a partial isometry.
\end{lemma}
\begin{proof} Let $a$ be an atom.
If $b \subseteq a$ then either $b = a$ or $b = 0$.
In both cases, $b \leq a$.
Thus $a$ is a partial isometry.
\end{proof}

The following example shows that the concepts we have introduced are distinct.

\begin{example}{\em Consider the following category $C$:
\begin{center}
\leavevmode
\xymatrix{e & f \ar@/^/[l]^{a}    \ar@/_/[l]_{b}}
\end{center}
The monoid $\mathsf{P}(C)$ has 16 elements.
The element $\{a\}$ is a partial isometry;
the elements $\{a,b\}$ is bideterministic but not a partial isometry
(since $\{a\} \subseteq \{a,b\}$ but  $\{a\} \nleq \{a,b\}$);
the element $\{a,e\}$ is not bideterministic.}
\end{example}

\begin{lemma}\label{lem:products-of-pi}
Let $C$ be a category.
The product of partial isometries in $\mathsf{P}(C)$ is also a partial isometry.
\end{lemma}
\begin{proof} 
We use our characterization of partial isometries from Lemma~\ref{lem:det-codet-new}.
Let $A$ and $B$ be non-empty partial isometries.
We prove that $AB$ is a partial isometry.
Let $x,y \in AB$ and suppose that $\mathbf{d}(x) = \mathbf{d}(y)$.
Then $x = ab$ and $y = a'b'$ where $\exists ab$ and $\exists a'b'$ in the category $C$
and $a,a' \in A$ and $b,b' \in B$.
We have that $\mathbf{d}(b) = \mathbf{d}(b')$.
But $B$ is a partial isometry and so $b = b'$.
We may similarly show that $a = a'$ and so $x = y$.
The other case follows by symmetry.
\end{proof}

\begin{proposition}\label{prop:esemigroups-cats} Let $C$ be a category.
Then $\mathsf{P}(C)$ is a Boolean Ehresmann monoid with the order being subset-inclusion
in which the product of partial isometries is a partial isometry.
\end{proposition}
\begin{proof} This follows by
Lemma~\ref{lem:ordered-e-cat} and Lemma~\ref{lem:products-of-pi}.
\end{proof}

We now prove the second theorem of this paper which characterizes the monoids arising in Proposition~\ref{prop:esemigroups-cats}.

\begin{theorem}\label{them:TWO} 
Let $(S,U)$ be a finite Boolean Ehresmann monoid
in which the product of partial isometries is a partial isometry.
Then there is a finite category $C$ such that $S$ is isomorphic to $\mathsf{P}(C)$ as an Ehresmann monoid
and in such a way that the order on $S$ is isomorphically mapped to the order on 
$\mathsf{P}(C)$.
\end{theorem}
\begin{proof} Let $S$ be a Boolean Ehresmann monoid having the stated properties.
Our proof is in seven steps.
In the first four steps, we construct the category $C$.
To do this, we shall use the atoms of $S$.
Recall that in a finite Boolean algebra every element is a finite join 
of atoms and, by Lemma~\ref{lem:atoms-are-pi}, every atom is a partial isometry. \\

(1) {\em If $a$ is an atom then both $a^{+}$ and $a^{\ast}$ are atoms.}
Suppose that $a$ is an atom.
Let $e \subseteq a^{+}$, where $e$ is a projection by axiom (OE2).
Then $e \leq a^{+}$ by axiom (OE3).
By definition, $ea \leq_{r} a$ and so $ea \subseteq a$ by axiom (OE4).
But $a$ is an atom 
and so either $ea = a$ or $ea = 0$.
In the first case, $a^{+} \leq e$ and so $a^{+} \subseteq e$ by axiom (OE4).
We therefore deduce that $e = a^{+}$.
In the second case, $ea^{+} = 0$ by the properties of $+$ 
and Lemma~\ref{lem:basic-results}.
Thus $e = 0$.
We have therefore proved that $a^{+}$ is also an atom.
By symmetry, if $a$ is an atom, then so too is $a^{\ast}$.\\

(2) {\em If $a$ and $b$ are atoms then $ab \neq 0$ if and only if $a^{\ast} = b^{+}$.}
Suppose first that $a^{\ast} = b^{+}$.
If $ab = 0$ then $(ab)^{\ast} = 0$ and so by (ES4), we have that $(a^{\ast}b)^{\ast} = 0$
and so $b^{\ast} = 0$ since $a^{\ast} = b^{+}$.
It follows by Lemma~\ref{lem:basic-results} that $b = 0$ which is a contradiction since $b$ is an atom.
Suppose now that $ab \neq 0$.
Then $ab = (ab^{+})(a^{\ast}b)$ which is a restricted product.
Now $ab^{+ } \leq a$ and so $ab^{+} \subseteq a$ by axiom (OE4).
But $a$ is an atom.
Thus $ab^{+} = 0$ or $ab^{+} = a$.
But $ab \neq 0$ and so $ab^{+} = a$.
Similarly, $a^{\ast}b = b$.
It follows that $a^{\ast} = b^{+}$ and so the product is a restricted one.\\

(3) {\em If $a$ and $b$ are atoms and $ab$ is a restricted product then $ab$ is an atom.}
In (Step 2), we showed that $ab \neq 0$.
Both $a$ and $b$ are atoms and so by Lemma~\ref{lem:atoms-are-pi}
each of $a$ and $b$ is a partial isometry.
By assumption, their product $ab$ is a partial isometry.
Let $c \subseteq ab$.
Then $c \leq ab$ since $ab$ is a partial isometry.
Thus, in particular, $c = ab c^{\ast}$.
Now, $bc^{\ast} \leq b$ and so $bc^{\ast} \subseteq b$ by axiom (OE4).
But $b$ is an atom.
It follows that  $bc^{\ast} = 0$ or $bc^{\ast} = b$.
Suppose first that $bc^{\ast} = 0$.
Then $b^{\ast}c^{\ast} = 0$ by axiom (ES4).
But from $c \subseteq ab$ we get that $c^{\ast} \subseteq (ab)^{\ast} \leq b^{\ast}$
where we have used Lemma~\ref{lem:order-properties-new} 
and Lemma~\ref{lem:basic-results}.
Thus $c^{\ast} \leq b^{\ast}$ by axiom (OE3).
It follows that $c^{\ast} = 0$ and so $c = 0$ by Lemma~\ref{lem:basic-results}.
Now suppose that $bc^{\ast} = b$.
Then $c = ab$.
We have therefore proved that $ab$ is an atom.\\

(4) Let $a$ be an atom.
Define $\mathbf{d}(a) = a^{\ast}$ and $\mathbf{r}(a) = a^{+}$.
By (Step 1) above these are both atoms.
Put $C$ equal to the set of all atoms of $S$.
By (Step 3), the set $C$ is a category under the restricted product. 
We have therefore proved that $C$ is a category whose set of identities is the set of atoms in $U$.\\

We may accordingly construct the Ehresmann monoid $\mathsf{P}(C)$ whose set of projections is $\mathsf{P}(C_{o})$.
We prove that $S$ is isomorphic to $\mathsf{P}(C)$ as Ehresmann monoids.\\

(5) Since we are in a Boolean algebra, 
every non-zero element is a join of the atoms below it.
If $a$ is an element of $S$ define $\phi (0) = \varnothing$ and if $a \neq 0$ define $\phi (a)$ to be the set of all atoms below it. 
The map $\phi$ determines a bijection from $S$ to $\mathsf{P}(C)$.
Observe that by axiom (OE2), the elements below a projection are all projections.
Thus each projection is a finite join of atoms which are themselves projections.
It follows that $\phi$ also determines a bijection from $U$ to $\mathsf{P}(C_{o})$.
These bijections are actually isomorphisms of Boolean algebras.\\

(6) {\em By (Step 3) above, we have that  $\phi (a) \phi (b) \subseteq \phi (ab)$.
We prove the reverse inclusion.}
Let $x$ be an atom such that $x \subseteq ab$.
We are working in a Boolean algebra, 
and so
each of $a$ and $b$ can be written as unions of atoms.
Let 
$a = \bigcup_{i=1}^{m} a_{i}$ and $b = \bigcup_{j=1}^{n} b_{j}$, where $a_{i}$ and $b_{j}$ are atoms.
Then $ab = \bigcup_{1 \leq i \leq m, 1 \leq j \leq n} a_{i}b_{j}$ where we omit all products 
of atoms which are zero.
Thus by (Step 3), 
we know that each $a_{i}b_{j}$ is an atom and a restricted product.
We now use the distributivity property of Boolean algebras to deduce that
$x = \bigcup_{i,j} (x \wedge a_{i}b_{j})$.
Now $x$ is an atom and so $x \wedge a_{i}b_{j} = 0$ or $x \subseteq a_{i}b_{j}$.
Suppose that $x \subseteq a_{i}b_{j}$ for some $i$ and some $j$.
Then, since $a_{i}b_{j}$ is an atom, we must have that $x = a_{i}b_{j}$.
We have therefore written $x$ as a product of atoms where $a_{i} \subseteq a$
and $b_{j} \subseteq b$.\\

We have proved that $\phi$ is an isomorphism of semigroups
between $S$ and $\mathsf{P}(C)$ and between $U$ and $\mathsf{P}(C_{o})$.\\

(7) {\em $\phi$ is an isomorphism of Ehresmann semigroups.}
Let $a \in S$ be a non-zero element.
Then $a$ is the join of all the atoms below it.
Thus by axiom (OE5), we have that  $a^{\ast}$ is equal to the join of all the atoms of the form $e$ where $e \subseteq a^{\ast}$.
We now prove that every atom $e$ below $a^{\ast}$ is of the form $b^{\ast}$ 
where $b$ is an atom and $b \subseteq a$.
Let $e \leq a^{\ast}$ where $e$ is an atom.
Then $ae \leq_{l} a$ and so $ae \subseteq a$ by axiom (OE4).
Observe that $(ae)^{\ast} = e$.
If $ae$ is an atom then we are done.
If $ae$ is not an atom then $x \leq ae$ where $x$ is an atom since we are working in a Boolean algebra.
But $x^{\ast} \subseteq (ae)^{\ast} = e$.
But $e$ is an atom.
Thus either $x^{\ast} = 0$ which implies that $x = 0$,
which is ruled out since $x$ is an atom,
or $x^{\ast} = e$.
We have therefore found an atom $x \leq a$ such that $x^{\ast} = e$.
Consequently, $\theta (a^{\ast}) = \theta (a)^{\ast}$.
A dual result holds for $+$.
\end{proof}

\section{Restriction semigroups}

Recall that a restriction semigroup is an Ehresmann semigroup in which each element is bideterministic.
In this section, we shall, in effect, combine results we proved in the previous two sections.
The following results are all well-known.
We include them for the sake of completeness.

\begin{lemma}\label{lem:needed} Let $(S,U)$ be a restriction semigroup.
\begin{enumerate}
\item The partial orders $\leq_{l}$ and $\leq_{r}$
are the same and so each is equal to $\leq$.
\item The semigroup $S$ is partially ordered with respect to $\leq$.
\item If $a \leq bc$ then there exist $b' \leq b$ and $c' \leq c$ such that $b'c'$ is a restricted product and $a = b'c'$.
\item The set $U$ is an order-ideal of $S$.
\item If $a,b \leq c$ and $a^{\ast} = b^{\ast}$ (respectively, $a^{+} = b^{+})$ then $a = b$.
\item If $a \leq b$ then $a^{\ast} \leq b^{\ast}$ and $a^{+} \leq b^{+}$.
\end{enumerate}
\end{lemma}
\begin{proof} (1) Suppose that $a \leq_{r} b$. We prove that $a \leq_{l} b$;
the proof of the other direction follows by symmetry.
We are given that $a = eb$ where $e$ is a projection.
But $eb = b(eb)^{\ast}$ since $b$ is deterministic.
It follows that $a = b(eb)^{\ast}$ and so $a \leq_{l} b$.

(2) This follows by (1) above because $\leq_{l}$ is left compatible with the multiplication
and $\leq_{r}$ is right compatible.

(3) Suppose that $a \leq bc$.
Then $a = (a^{+}b)c$.
Put $e = (a^{+}b)^{\ast}c^{+}$.
Then $a = (a^{+}be)(ec)$.
Put $b' = a^{+}be$ and $c' = ec$.
Then $a = b'c'$, the product $b'c'$ is a restricted product,
and $b' \leq b$ and $c' \leq c$.

(4) This is immediate.

(5) The proof of this is straightforward.

(6) This follows by the properties of the partial orders.
\end{proof}

Let $C$ be a category.
Denote the set of partial isometries in $\mathsf{P}(C)$ by $\mathscr{PI}(C)$.
This is a restriction monoid by Lemma~\ref{lem:ordered-e-cat}, Lemma~\ref{lem:products-of-pi}, and Lemma~\ref{lem:cake}.

\begin{theorem}\label{them:THREE} Let $(S,U)$ be a restriction semigroup.
Then there is a category $C$ and an injective morphism $\alpha \colon S \rightarrow \mathscr{PI}(C)$.
\end{theorem}
\begin{proof} For the category $C$, we take the set $S$ equipped with the restricted product
according to Proposition~\ref{prop:cats}.
Then $\mathscr{PI}(C)$ is a restriction monoid.
Define $\alpha \colon S \rightarrow \mathscr{PI}(C)$ 
by putting $\alpha (a)$ equal to the set of all elements less than or equal to $a$.
This is well-defined by part (5) of Lemma~\ref{lem:needed} and injective.
By part (4) of Lemma~\ref{lem:needed} elements of $U$ are mapped to elements of $\mathsf{P}(C_{o})$.
It is a homomorphism by part (3) of Lemma~\ref{lem:needed}.
It remains to check that it is a morphism.
By symmetry, it is enough to check the morphism property for $\ast$.
If $b \leq a$ then $b^{\ast} \leq a^{\ast}$ by part (6) of Lemma~\ref{lem:needed}.
On the other hand, if $e \leq a^{\ast}$ then $ae \leq a$ and $(ae)^{\ast} = a^{\ast}$.
\end{proof}

The following is related to questions discussed in \cite{GK}, although our approach is quite different. 
A natural question is the following:
given a restriction semigroup $S$, under what circumstances 
can $S$ be embedded into an inverse semigroup $T$ 
in such a way that $a^{\ast} = a^{-1}a$ and $a^{+} = aa^{-1}$.
We call such an embedding of a restriction semigroup in an inverse semigroup
a {\em nice embedding}. 
This problem is answered below modulo the problem of embedding categories into groupoids.  

\begin{theorem}\label{them:FOUR} Let $S$ be a restriction semigroup.
Denote the set $S$ equipped with the restricted product by $C$.
Then $S$ admits a nice embedding 
if and only if the 
category $C$ can be embedded into a groupoid $G$.
\end{theorem}
\begin{proof} Suppose first that the category $C$ can be embedded into a groupoid $G$.
By Theorem~\ref{them:THREE}, there is an injective morphism $\alpha \colon S \rightarrow  \mathscr{PI}(C)$.
Let $C \subseteq G$.
Then $\mathscr{PI}(C) \rightarrow \mathscr{PI}(G)$
is an injective morphism which is actually an embedding.
But the elements of $\mathscr{PI}(G)$ are simply the local bisections of $G$ and so form an inverse semigroup;
see Remark~\ref{rem:trump}.
It follows that $S$ admits a nice embedding into the inverse semigroup $\mathscr{PI}(G)$. 
To prove the converse, suppose that $S$ can be nicely embedded into the inverse semigroup $T$.
Let $G$ be the set $T$ equipped with the restricted product.
Then $C$ embeds into $G$ as a subcategory.
\end{proof}


\end{document}